\documentclass[12pt,a4paper]{article}
\usepackage{amsmath}
\usepackage{amsthm}
\usepackage{tikz}
\usepackage{amsfonts}
\usepackage{multirow}
\usepackage{tikz-cd}
\usepackage{subcaption}
\usetikzlibrary{arrows,shapes}
\usetikzlibrary{matrix}
\usepackage{float}
\usepackage{algorithm}
\usepackage{algpseudocode}
\usepackage{hyperref}
\usepackage{tabularx}
\usepackage{amssymb}
\usepackage{graphicx}
\usepackage{tcolorbox}
\usepackage{booktabs}
\usepackage{rotating,tabularx}

\usepackage[left=2cm,right=2cm,top=2cm,bottom=2cm]{geometry}
\newcolumntype{L}[1]{>{\raggedright\let\newline\\\arraybackslash\hspace{0pt}}m{#1}}
\newcolumntype{C}[1]{>{\centering\let\newline\\\arraybackslash\hspace{0pt}}m{#1}}
\newcolumntype{R}[1]{>{\raggedleft\let\newline\\\arraybackslash\hspace{0pt}}m{#1}}

\newtheorem{Theorem}{Theorem}[section]
\newtheorem{Proposition}[Theorem]{Proposition}
\newtheorem{Remark}[Theorem]{Remark}
\newtheorem{Lemma}[Theorem]{Lemma}

\newtheorem{Definition}[Theorem]{Definition}

\usepackage{hyperref}
\hypersetup{colorlinks=true,urlcolor=blue}
\expandafter\let\expandafter\oldproof\csname\string\proof\endcsname
\let\oldendproof\endproof
\renewenvironment{proof}[1][\proofname]{
\oldproof[\ttfamily\scshape \bf #1.]
}{\oldendproof}
\def\ve{\varepsilon}

\def\emp{\emptyset}

\def\dom{{\rm dom}\,}

\def\ox{\overline{x}}
\def\oy{\overline{y}}

\def\Bar{\overline}
\def\ra{\rangle}
\def\la{\langle}
\def\ve{\varepsilon}
\def\epsilon{\varepsilon}
\def\ox{\bar{x}}
\def\oy{\bar{y}}

\def\dom{\mbox{\rm dom}\,}

\def\ph{\varphi}
\def\emp{\emptyset}

\def\oR{\Bar{\R}}
\def\lm{\lambda}

\def \N{{\rm I\!N}}
\def \R{{\rm I\!R}}

\numberwithin{equation}{section}

\begin{document}

\title{\bf Local Minimizers of Nonconvex Functions\\ in Banach Spaces via Moreau Envelopes}
\author{Pham Duy Khanh\footnote{Department of Mathematics, Ho Chi Minh City University of Education, Ho Chi Minh City, Vietnam. E-mail: pdkhanh182@gmail.com} \quad Vu V.H. Khoa\footnote{Department of Mathematics, Wayne State University, Detroit, Michigan, USA. E-mail: khoavu@wayne.edu. Research of this author was partly supported by the US National Science Foundation under grant DMS-2204519.}\quad Boris S. Mordukhovich\footnote{Department of Mathematics, Wayne State University, Detroit, Michigan, USA. E-mail: aa1086@wayne.edu. Research of this author was partly supported by the US National Science Foundation under grants DMS-1808978 and DMS-2204519, by the Australian Research Council under Discovery Project DP-190100555, and by Project~111 of China under grant D21024.}\quad Vo Thanh Phat\footnote{Department of Mathematics, Wayne State University, Detroit, Michigan, USA. E-mail: phatvt@wayne.edu. Research of this author was partly supported by the US National Science Foundation under grants DMS-1808978 and DMS-2204519.}}
\maketitle
\begin{center}
{\bf To Tam\'as Terlaky in honor of his 70th birthday\\ with friendship and great respect}
\end{center}

\medskip
\begin{quote}
\noindent {\bf Abstract.} 
This paper investigates the preservation of local minimizers and strong minimizers of extended-real-valued lower semicontinuous functions under taking their Moreau envelopes. We address a general setting of Banach spaces, while all the obtained results are new even for functions in finite dimensions. Our main motivation came from applications to numerical methods in nonsmooth optimization dealing with broad classes of nondifferentiable and nonconvex functions. The paper also formulates and discusses some open questions stemming from this study.


\end{quote}

\section{Introduction}
The concept of Moreau envelopes, introduced by Jean-Jacques Moreau in \cite{Moreau}, has been recognized as a fundamental concept in convex and variational analysis, optimization, and their numerous applications. This concept has been well understood as an efficient tool to study both theoretical and numerical aspects of analysis and its applications.\vspace*{0.03in}

Moreau envelopes were first defined and explored for extended-real-valued convex functions and then were expanded to a general class of lower semicontinuous (l.s.c.) functions in finite-dimensional and infinite-dimensional spaces. It has been confirmed that this concept provides an effective {\em regularization} of the given function to get certain desired properties that include continuity, differentiability, full domain, etc. For further explorations of Moreau envelopes in these and other contexts of convex and variational analysis, we refer the reader to the books \cite{Bauschke2011,Correa,Rockafellar98,Thibault} and their extended bibliographies. 

From the viewpoint of optimization, a crucial property of Moreau envelopes is their ability to preserve the global minimum value and the set of {\em global minimizers}; see \cite[Example~1.46]{Rockafellar98} with the proof and references therein. In the local setting, an optimization relationship between a given function and its Moreau envelope was established for the particular type of solutions known as {\em tilt-stable local minimizers}. This remarkable notion was introduced by Poliquin and Rockafellar \cite{Poli} for extended-real-valued functions on finite-dimensional spaces and then developed by Mordukhovich and Nghia \cite{nghia} in the Hilbert space setting. By now, tilt-stable minimizers have been largely investigated and characterized for various classes of constrained optimization problems in both finite and infinite dimensions. We refer the reader to the book by the third author \cite{Mordukhovich24} for a full account on tilt stability and its applications.

The aforementioned relationship for tilt-stable minimizers was established by Mordukhovich and Sarabi in the proof of \cite[Theorem~6.2]{BorisEbrahim} for the special class of {\em prox-regular} and {\em subdifferentially continuous} functions on $\R^n$, which play a pivoting role in second-order variational analysis. This result says that if $\ox$ is a tilt-stable local minimizer for such a function, then it is a tilt-stable minimizer for its Moreau envelope. The obtained result reduces finding tilt-stable minimizers of extended-real-valued prox-regular functions to determining such minimizers of their Moreau envelopes, which are continuously differentiable with Lipschitz continuous gradients. This scheme has been implemented in \cite{BorisKhanhPhat,kmptjogo,kmptmp,Mordukhovich24,BorisEbrahim} for developing generalized Newton algorithms to find tilt-stable local minimizers as well as global solutions to various classes of optimization and related problems in finite-dimensional spaces.\vspace*{0.03in} 

Apart of  this, we are not familiar with any relationships between usual (not tilt-stable) {\em local minimizers} of a given function and its Moreau envelope, even in finite dimensions. Establishing such relationships is a challenging issue being instrumental to conduct {\em local convergence analysis} of {\em nonconvex} optimization problems by employing the aforementioned scheme for {\em generalized  Newtonian algorithms}. Furthermore, the {\em inexact proximal point method} for a given function is essentially the {\em inexact gradient descent method} for its Moreau envelope; see \cite{kmpt23.3,Dat,kmt23.2}. Therefore, exploring the preservation of local minimizers for Moreau envelopes provides insights into conducting local convergence  analysis of the inexact proximal point method and its variants from the perspectives of the inexact gradient descent method and its modifications.\vspace*{0.03in}

The {\em primary goal} of this paper is to establish the {\em equivalence} between {\em arbitrary local minimizers} of a given l.s.c.\ prox-bounded function and its Moreau envelope under general unrestrictive assumptions in {\em reflexive Banach spaces}. Moreover, such an equivalence will also be established for {\em strong local minimizers} with some quantitative modulus relationships in the setting of {\em Hilbert spaces}. To achieve these aims, we use powerful tools of variational analysis married to geometric theory of Banach spaces.\vspace*{0.03in}

The subsequent content is organized as follows. Section~\ref{sec:prelim} presents some preliminary material from variational analysis and optimization needed for the formulations and proofs of the main results. Section~\ref{sec:localminimum} revolves around the relationships between local minimizers of functions and their Moreau envelopes. Section~\ref{sec:strong} is devoted to the study of such relationships for strong local minimizers. The final Section~\ref{sec:open} formulates some open questions for the future research.

\section{Preliminaries}\label{sec:prelim}
In this section, we define and discuss some notions from variational analysis and generalized differentiation that are frequently used in what follows. More details and references can be found in the books \cite{Mordukhovich24,Rockafellar98,Thibault} in finite and infinite dimensions. We also recall the notions of local and strong local minimizers that are investigated in this paper. For the notions and results from geometry of Banach spaces, we refer the reader to the book \cite{fabian}. 

Unless otherwise stated, the spaces under consideration are assumed to be {\em Banach} with the generic symbols $\|\cdot\|$ for norms and $\langle\cdot,\cdot\rangle$ for the canonical pairing between a space $X$ and its topological dual $X^*$. We use `$\rightarrow$' and `$\overset{w}{\rightarrow}$' to indicate the strong/norm and weak convergence, respectively. The symbols $\mathbb{B}_{r}(x)$ and $B_r (x)$ signify the closed and open balls centered at $x$ with radius $r>0$, respectively. Given $C\subset X$ and $x\in X$, the {\em distance function to} $C$ is defined as 
\begin{equation}\label{dist}
d(x;C):= \inf_{y\in C} \|x-y\|\;\mbox{ for all }\;x\in X.
\end{equation}
Given an l.s.c.\ function $\ph\colon X\to\oR$ on a Banach space $X$,
the {\em Moreau envelope} of $\ph$ is
\begin{equation}\label{el}
e_\lambda \varphi(x) := \inf_{w\in X}\left\{\varphi(w) + \frac{1}{2\lambda}\|w-x\|^2\right\},\quad x\in X,
\end{equation}   
where $\lm>0$ is a parameter. The corresponding {\em proximal mapping} is
\begin{equation}\label{pl}
P_\lambda \varphi(x) := \text{\rm argmin}_{w\in X}\left\{ \varphi(w) + \frac{1}{2\lambda}\|w-x\|^2\right\},\quad x\in X.
\end{equation}
As in \cite{Thibault}, we say that $\ph\colon X\to\oR$ is {\em prox-bounded} if 
\begin{equation}\label{prox-bounded}
\varphi(x)\ge\alpha\|x-\ox\|^2+\beta\;\mbox{ for some }\;\alpha,\beta\in \R\;\mbox{ and }\;\ox\in X.
\end{equation}
Recall further that $\ox^*\in X^*$ is a {\em proximal subgradient} of $\ph$ at $\ox\in\dom\ph$ if there exist positive numbers $r$ and $\varepsilon$ for which
\begin{equation}\label{fixed1}
\varphi(x)\ge \varphi(\ox) + \la \ox^*, x - \ox\ra -\frac{r}{2}\|x - \ox\|^2\;\mbox{ whenever }\;x \in B_{\epsilon}(\ox).
\end{equation}
The collection of all proximal subgradients of $\varphi$ at $\ox$ is called the {\em proximal subdifferential} of the function
at this point and is denoted by $\partial_P\varphi(\ox)$.
  
\medskip 
We say that a function $\varphi:X\rightarrow \overline{\R}$ is {\em weakly sequentially lower semicontinuous} (weakly sequentially l.s.c.) at $\ox \in \dom \varphi$ if for any sequence $\{x_k\}$ which weakly converges to $\ox$, it holds that $\liminf_{k\rightarrow \infty}\varphi (x_k) \ge \varphi (\ox)$. This function is called {\em weakly sequentially l.s.c.} if the preceding property holds for all $\ox\in \dom \varphi$. It easily follows from the definition that $\varphi$ is weakly sequentially l.s.c. if and only if the level sets $\{x\in X\mid \varphi (x) \le \alpha\}$ are weakly sequentially closed in $X$ for all $\alpha\in \R$. We clearly have that the weak sequential lower semicontinuity of a function implies its lower semicontinuity, while not vice versa. The weak sequential l.s.c.\ of an extended-real-valued function $\varphi:X\rightarrow \overline{\R}$ together with its {\em coercivity}
\begin{equation}\label{coer}
|\varphi(x)|\rightarrow \infty \ \text{ whenever }\ \|x\|\rightarrow \infty
\end{equation}
guarantees the existence of its {\em global minimizer} according to the following version of the Weierstrass existence theorem taken from \cite[Theorem~7.3.7]{FA}.

\begin{Proposition}\label{minimizer} 
Let $X$ be a reflexive Banach space, and let $\varphi: X \rightarrow \mathbb{R}$ be a coercive and weakly sequentially lower semicontinuous function. Then there exists $x_0 \in X$ such that
$$
f\left(x_0\right)=\inf _{x \in X} f(x).
$$
\end{Proposition}

We conclude this section with recalling the notions of local minimizers and strong local minimizers considered in what follows.

\begin{Definition}\label{minimizers} \rm Let $\varphi:X\to\overline{\R}$, and let $\ox \in \dom \varphi$. We say that:

{\bf(i)}  $\ox$ is a {\em local minimizer} of $\varphi$ if there exists $\epsilon>0$ such that
$$
\varphi(x) \geq \varphi(\ox) \quad \text{for all }\; x \in B_\epsilon(\ox). 
$$
 
{\bf (ii)} $\ox$ is a {\em strong local minimizer} of $\varphi$ with modulus $\sigma>0$ if there exists $\epsilon>0$ such that
$$
\varphi(x) \geq  \varphi(\ox) +\frac{\sigma}{2}\|x-\ox\|^2 \quad \text{for all }\; x \in B_\epsilon(\ox).  
$$
\end{Definition}

\section{Relationships for Local Minimizers}\label{sec:localminimum}

In this section, we establish relationships between local minimizers of a function $\ph\colon X\to\oR$ and of its Moreau envelope $e_\lm\ph$ in Banach spaces. To proceed, we first need to recall some known results used below. Two following lemmas  are taken from \cite[Propositions~6.1 and 6.2]{KKMP23-2}. 

\begin{Lemma}\label{prop:charac-prox-bound}
$\ph\colon X\to\oR$ be an l.s.c.\ function defined on a Banach space $X$. Then we have the three equivalent assertions:

{\bf(i)} $\varphi$ is prox-bounded in the sense of \eqref{prox-bounded}.

{\bf(ii)} For some $\lambda >0$ and $x\in X$, it holds that $e_{\lambda} \varphi (x) >-\infty$.

{\bf(iii)} There is $\lambda_0 >0$ such that $e_{\lambda}\varphi (x)>-\infty$ for any $0<\lambda<\lambda_0$ and any $x\in X$.
\end{Lemma}
\noindent
The positive constant $\lambda_0$ in (iii) is usually called the {\em threshold of prox-boundedness} of $\varphi$.

\begin{Lemma}\label{lem:Patxbar} Let $\varphi:X\to \oR$ be an l.s.c.\ function on a Banach space $X$, and let $\ox\in \dom\varphi$. Then we have the three equivalent assertions:

{\bf(i)} $\varphi$ is prox-bounded, and $0\in X^*$ is a proximal subgradient of $\varphi$ at $\ox$.

{\bf(ii)} $P_{\lambda} \varphi (\ox)=\{\ox\}$ for some $\lambda >0$.

{\bf(iii)} $P_{\lambda} \varphi (\ox)=\{\ox\}$ for all $\lm>0$ sufficiently small.
\end{Lemma}
 
To establish the main result of this section, we need the following proposition of its own interest about a nice behavior of proximal mappings in reflexive Banach spaces. 

\begin{Proposition}\label{continuity-prox}
Let $\varphi: X \rightarrow \overline{\mathbb{R}}$ be a proper, prox-bounded, and weakly sequentially l.s.c.\ function defined on a reflexive Banach space $X$. Fixing any $0<\lambda <\lambda_{\varphi}$, we have the assertions:

{\rm {\bf (i)}} $P_{\lambda}\varphi (x)\neq \emptyset$ for any $x\in X$.

{\rm {\bf (ii)}} If $w^k \in P_\lambda \varphi(x^k)$ for $k\in \N$ such that $x^k \rightarrow \bar{x}$ as $k\to\infty$, then the sequence $\left\{w^k\right\}_{k \in \mathbb{N}}$ is bounded and all its weak accumulation points lie in $P_\lambda \varphi(\ox)$.
\end{Proposition}
\begin{proof}
First we verify (i). Fix any $x\in X$ and claim that $\varphi (\cdot) + \frac{1}{2\lambda}\|x-\cdot\|^2$ is coercive in the sense of \eqref{coer}, which then opens the door to utilizing Proposition~\ref{minimizer}. To proceed, choose any $\lambda_1 \in (\lambda, \lambda_\varphi)$ and deduce from Lemma~\ref{prop:charac-prox-bound}(iii) that $\beta:=e_{\lambda_1}\varphi (0) > -\infty$. The definition of $\beta$ implies that $\varphi (w) + \frac{1}{2\lambda_1}\|w\|^2\ge \beta$ for all $w\in X$. Therefore, we get the estimates
\begin{eqnarray}\label{eq:coercive-1}
\varphi (w) + \dfrac{1}{2\lambda}\|w-x\|^2 &\ge& -\dfrac{1}{2\lambda_1}\|w\|^2 + \dfrac{1}{2\lambda}\|w-x\|^2 + \beta \notag \\
&\ge& \left(\dfrac{1}{2\lambda}-\dfrac{1}{2\lambda_1}\right)\|w\|^2 -\dfrac{1}{\lambda}\|w\|\cdot \|x\| + \dfrac{1}{2\lambda}\|x\|^2 + \beta 
\end{eqnarray}
for all $w\in X$, where the triangle inequality is used in \eqref{eq:coercive-1}. As the leading term $\frac{1}{2\lambda}-\frac{1}{2\lambda_1}$ is positive whenever $\lambda_1>\lambda$, the right-hand side of \eqref{eq:coercive-1} tends to $\infty$ as $\|w\|\rightarrow \infty$ which verifies the coercivity of $\varphi (\cdot)+\frac{1}{2\lambda}\|x-\cdot\|^2$ on $X$. Since both $\varphi$ and the norm-square function $\|x-\cdot\|^2$ are weakly sequentially l.s.c., their sum is weakly sequentially l.s.c.\ as well. Thus it follows from Proposition~\ref{minimizer} that $P_{\lambda}\varphi (x)\neq \emptyset$, which justifies assertion (i).\vspace*{0.03in}

To verify (ii), suppose that there exist sequences $x^k\rightarrow \ox$ and $\{w_k\}\subset X$ for which $w^k \in P_\lambda \varphi(x^k)$ as $k\in \N$. Pick any $\alpha > e_{\lambda}\varphi (\ox)$. By the global continuity of $e_{\lambda}\varphi$ on $X$ proved in \cite[Proposition~3.3]{JTZ} and the fact that $x^k \rightarrow \ox$ as $k\to\infty$, it follows from $\alpha > e_{\lambda}\varphi (\ox)$ that $\alpha > e_{\lambda} \varphi (x^k)$ for sufficiently large $k\in \N$. The latter can be rewritten by the fact that $w^k\in P_\lambda \varphi(x^k)$ for all $k\in \N$ as follows:
\begin{equation}\label{eq:alpha-F}
{\alpha > e_\lambda \varphi(x^k) = \varphi(w^k)+ \frac{1}{2\lambda}\|w^k-x^k\|^2 \quad \text{for large }\;k.}
\end{equation}
Combining \eqref{eq:coercive-1} and \eqref{eq:alpha-F} for $w:=w^k$  whenever $k\in \N$ tells us that
\begin{equation}\label{trinomial}
\left(\dfrac{1}{2\lambda}-\dfrac{1}{2\lambda_1}\right)\|w^k\|^2 -\dfrac{1}{\lambda}\|w^k\|\cdot \|x^k\| + \dfrac{1}{2\lambda}\|x^k\|^2 + \beta - \alpha < 0 \quad \text{ for large }k\in \N.
\end{equation}
Since $\lambda_1>\lambda$ and the sequence $\{x^k\}$ converges (and hence is bounded), the inequality in \eqref{trinomial} implies that the sequence $\{w^k\}$ is bounded as well.

Remembering that the space $X$ is reflexive, we suppose without loss of generality the weak convergence of $(w^k,x^k)$ to $(\hat{w},\ox)$ for some $\hat{w}\in X$. Observe further that the inequality $\varphi (\hat{w}) + \frac{1}{2\lambda}\|\hat{w}-\ox\|^2\le \alpha$ holds. Indeed, the weak sequential l.s.c.\ of $\varphi$ and $\|\cdot\|^2$ yields the weak sequential l.s.c.\ of $\varphi (\cdot) + \frac{1}{2\lambda}\|\cdot - \cdot\|^2$ with respect to $(w,x)$. The latter implies that $\varphi (\hat{w}) + \frac{1}{2\lambda}\|\hat{w}-\ox\|^2\le \alpha$ by letting $k\rightarrow \infty$ in \eqref{eq:alpha-F}. Finally, letting $\alpha \searrow e_{\lambda}\varphi (\ox)$ brings us to
$$
\varphi (\hat{w}) + \dfrac{1}{2\lambda}\|\hat{w}-\ox\|^2  \le e_{\lambda}\varphi (\ox),
$$
i.e., $\hat{w}\in P_\lambda \varphi(\ox)$, which justifies (ii) and completes the proof of the proposition.
\end{proof}

\begin{Remark}\rm  The nonemptiness of the proximal values $P_{\lambda}\varphi(x)$ for $x\in X$ established in Proposition~\ref{continuity-prox}(i) plays an important role in the proof of our main results presented below. In our proof, this crucial property is obtained by using merely the prox-boundedness and weak sequential l.s.c.\ of the function defined on a reflexive Banach space. Observe to this end that the paper \cite{Penot98} contains various sufficient conditions on the function and the space in question to ensure that $P_{\lambda}\varphi (x)\neq \emptyset$. Namely, \cite[Corollary~3.7]{Penot98} requires $X$ to be $M$-reflexive and Fr\'echet smooth in order to ensure the density of the set $\{x\in X\mid P_{\lambda}\varphi (x)\neq \emptyset\}$. Comparing with our requirements, the result of \cite{Penot98} relaxes the weak sequential l.s.c.\ of $\varphi$, while demands on the other hand the $M$-reflexivity and Fr\'echet smoothness of $X$. It is worth noting that since the definition of Moreau envelopes and proximal mappings depend heavily on the norm under consideration, renorming procedures may not be useful.
\end{Remark}

The next theorem is the main outcome in this paper establishing close relationships between local minimizers of the function in question and of its Moreau envelope in Banach spaces.

\begin{Theorem}\label{localMoreau}  Let $\varphi: X \to \overline{\R}$ be an l.s.c.\ function defined on a Banach space $X$, and let $\ox\in \dom \varphi$. Assume that $\varphi$ is prox-bounded with threshold $\lambda_{\varphi}>0$, and let $0<\lambda<\lambda_{\varphi}$. Consider the following assertions:

{\bf(i)} $\ox$ is a local minimizer of $\varphi$.

{\bf(ii)} $\bar{x}$ is a local minimizer of $e_\lambda\varphi$.\\
Then we have the implication {\bf(i)}$\Longrightarrow${\bf(ii)} provided that $X$ is reflexive and that $\varphi$ is weakly sequentially lower semicontinuous. Moreover, in this case there exists a neighborhood $U_{\lambda}$ of $\ox$ such that the lower error bound for the Moreau envelope 
\begin{equation}\label{Moreauineq}
e_\lambda\varphi(x) - e_\lambda\varphi(\bar{x}) \geq  \frac{1}{2\lambda}d^2\big(x;P_{\lambda}\varphi(x)\big) \quad \text{for all} \quad x \in U_\lambda
\end{equation} 
holds via the distance function \eqref{dist}. The inverse implication {\bf(ii)}$\Longrightarrow${\bf(i)} is fulfilled if $P_\lambda \varphi(\ox)$ is nonempty, which is surely the case when $X$ is reflexive and $\varphi$ is weakly sequentially l.s.c.
\end{Theorem}
\begin{proof} First we verify implication (i)$\Longrightarrow$(ii) and estimate \eqref{Moreauineq} under the weak sequential l.s.c. of the function $\varphi$ and the reflexivity of the Banach space $X$. Suppose that $\bar{x}$ is a local minimizer of $\varphi$ and find a neighborhood $U$ of $\bar{x}$ such that
\begin{equation}\label{local}
\varphi(x)\geq \varphi(\bar{x}) \quad \text{for all} \quad x \in U.
\end{equation}
As $\ox$ is a local minimizer of $\varphi$, the origin $0\in X^*$ is a proximal subgradient of $\varphi$ at $\ox$. Lemma~\ref{lem:Patxbar} tells us that whenever $0<\lambda<\lambda_{\varphi}$ we have $P_\lambda \varphi(\bar{x})=\{\bar{x}\}$. Supposing, contrary to the existence of some $U_{\lambda}$ guaranteeing \eqref{Moreauineq}, that inside any neighborhood of $\ox$ there exists a point $x$ for which $e_\lambda\varphi(x) - e_\lambda\varphi(\bar{x}) <  \frac{1}{2\lambda}d^2\big(x;P_{\lambda}\varphi(x)\big)$. This allows us to construct a sequence $\{x_k\}$ that converges to $\ox$ and satisfies the inequality
\begin{equation}\label{contra-4}
e_\lambda\varphi(x_k) - e_\lambda\varphi(\bar{x}) <  \dfrac{1}{2\lambda}d^2\big(x_k;P_{\lambda}\varphi(x_k)\big)\;\mbox{ for all }\;k\in \N.
\end{equation}
Take any $k\in \N$ and deduce from Proposition~\ref{continuity-prox}(i) that $P_{\lambda}\varphi(x_k)\ne\emp$. The latter allows to pick some $w_k\in P_{\lambda}\varphi (x_k)$ for all $k\in \N$. Combining this with $x_k\rightarrow \ox$ as $k\to\infty$, Proposition~\ref{continuity-prox}(ii) tells us there is a subsequence $\{w_{k_m}\}_{m\in \N}$ of $\{w_k\}_{k\in \N}$ ensuring the weak convergence $w_{k_m}\xrightarrow{w}\ox$ by taking into account that $P_\lambda \varphi(\bar{x})=\{\bar{x}\}$. Since $\varphi$ is weakly sequentially l.s.c.\ at $\ox$ and $e_{\lambda}\varphi$ is continuous at $\ox=\lim_{m\rightarrow \infty}x_{k_m}$, we have the relationships
\begin{eqnarray*}
e_\lambda\varphi(\bar{x})\leq \varphi(\bar{x}) &\leq&  \liminf_{m\rightarrow \infty}\varphi(w_{k_m})\\
&=&\liminf_{m\rightarrow \infty} \left(e_{\lambda}\varphi(x_{k_m})-\frac{1}{2\lambda}\|x_{k_m}-w_{k_m}\|^2\right) \\
&=& e_{\lambda}\varphi (\ox)- \dfrac{1}{2\lambda}\limsup_{m\rightarrow \infty}\|x_{k_m}-w_{k_m}\|^2,
\end{eqnarray*}
which imply the limiting conditions 
$$
\limsup_{m\to\infty}\|x_{k_m}-w_{k_m}\|^2=0\;\mbox{ and hence }\;
\lim_{m\rightarrow \infty}\|x_{k_m}-w_{k_m}\|^2=0.
$$
Since $x_{k_m}\rightarrow \ox$ as $m\to\infty$, it turns out that $w_{k_m}\rightarrow \ox$ as well. Therefore, for sufficiently large $m\in \N$ we have $w_{k_m}\in U$ and thus get by \eqref{local} the estimates
\begin{eqnarray*}
e_\lambda\varphi(\bar{x})\leq \varphi(\bar{x})
\leq\varphi(w_{k_m})=e_\lambda\varphi(x_{k_m})-\frac{1}{2\lambda}\|x_{k_m}-w_{k_m}\|^2 \le e_\lambda\varphi(x_{k_m})-\frac{1}{2\lambda}d^2\big(x_{k_m};P_{\lambda}\varphi (x_{k_m})\big),
\end{eqnarray*}
i.e., $e_{\lambda}\varphi (x_{k_m})-e_{\lambda}\varphi (\ox) \ge \frac{1}{2\lambda}d^2\big(x_{k_m};P_{\lambda}\varphi (x_{k_m})\big)$ for large $m$. This contradicts \eqref{contra-4}, and so \eqref{Moreauineq} holds. Finally, $\bar{x}$ is a local minimizer, which justifies the first part of the theorem.\vspace*{0.03in}

To verify the reverse implication (ii)$\Longrightarrow$(i), suppose that $P_{\lambda}\varphi (\ox)\neq \emptyset$ and pick $\oy\in P_{\lambda}\varphi (\ox)$. Since $\ox$ is a local minimizer of $e_{\lambda}\varphi$, there exists a neighborhood $U$ of $\ox$ on which $e_{\lambda}\varphi (x) \ge e_{\lambda}\varphi (\ox)$. Due to the definition of $e_{\lambda}\varphi$, this amounts to saying that
\begin{equation}\label{eq:2nd-way}
\varphi (y) + \dfrac{1}{2\lambda}\|x-y\|^2 \ge \varphi (\oy) + \dfrac{1}{2\lambda}\|\ox-\oy\|^2\;\mbox{ for all }\;x\in U \;\text{ and } \; y \in X.
\end{equation}
In \eqref{eq:2nd-way}, by choosing $y:=\oy$ and $x:=\ox + \varepsilon (\oy-\ox)$ with $\ve>0$ small enough to ensure that $x\in U$, we arrive at the estimate
\begin{equation*}
\|\ox+\varepsilon(\oy-\ox)-\oy\|^2 \ge\|\ox-\oy\|^2,
\end{equation*}
which can be equivalently rewritten as
\begin{equation*}
(1-\varepsilon)^2\|\ox-\oy\|^2 \ge \|\ox-\oy\|^2,\;\mbox{ i.e., }\;(\varepsilon-2)\|\ox-\oy\|^2\ge 0.
\end{equation*}
This obviously implies that $\ox=\oy$, and hence 
$P_\lambda \varphi(\ox)= \{\ox\}$. Then 
$$
\varphi(x) \geq e_\lambda\varphi(x) \geq e_\lambda\varphi(\ox)= \varphi(\ox) + \frac{1}{2\lambda}\|\ox-\ox\|^2 = \varphi(\ox) \quad \text{for all } \; x \in U, 
$$
which justifies (i) and thus completes the proof of the theorem. 
\end{proof}

\section{Relationships for Strong Local Minimizers}\label{sec:strong}

Now we proceed with establishing {\em qualitative and quantitative} relationships between {\em strong local minimizers} of extended-real-valued functions and their Moreau envelopes in {\em Hilbert spaces}. Given an l.s.c.\ function $\varphi: X\to \oR $ with $\ox \in \dom \varphi$ and given a modulus $\sigma\ne 0$, consider the {\em quadratically $\sigma$-shifted function} associated with $\ph$ by 
\begin{equation}\label{shift}
\psi(x):=\varphi(x) -\frac{\sigma}{2}\|x-\ox\|^2 \ \text{ for all } \ x\in X.
\end{equation}  
Recall first the following lemma taken from \cite[Proposition~6.8]{KKMP23-2}, which provides precise relationships between Moreau envelopes of a given function and its quadratic shift.

\begin{Lemma}\label{aPhat}
Let $\varphi: X\rightarrow\overline{\R}$ be a prox-bounded and l.s.c.\  function defined on a Hilbert space $X$, and let $\psi$ be the corresponding $\sigma$-shift with $\sigma\ne 0$. Then 
we have the equalities
\begin{equation}\label{Moreau-shift1}
e_{\lambda}\varphi (x) = e_{\lambda/(1+\sigma\lambda)}\psi \left( \dfrac{x+\sigma\lambda \ox }{1+\sigma\lambda}\right)+\dfrac{\sigma}{2(1+\sigma\lambda)}\|x-\ox\|^2,
\end{equation}
\begin{equation}\label{envelopepsi}
e_\lambda\psi(x)=e_{{\lambda}/(1-\sigma\lambda)}\varphi\left(\frac{x -\sigma\lambda \ox}{1-\sigma\lambda}\right)-\frac{\sigma}{2(1-\sigma\lambda)}\|x-\ox \|^2
\end{equation} 
for all $\lambda\in(0,|\sigma|^{-1})$ and $x\in X$.
\end{Lemma}

The next statement obviously confirms that strong local minimizers of l.s.c.\ functions reduce to the usual ones for their quadratic shifts \eqref{shift}.

\begin{Proposition}\label{localminishift} Let $\varphi: X\rightarrow\overline{\R}$ be an l.s.c.\  function defined on a Hilbert space $X$, and let $\psi$ be the corresponding $\sigma$-shift with $\sigma>0$. Then $\ox$ is a strong local minimizer of $\varphi$ with modulus $\sigma$ if and only if $\ox$ is a local minimizer of  the quadratically shifted function $\psi$ defined in \eqref{shift}. 
\end{Proposition} 
\begin{proof} According to Definition~\ref{minimizers}(ii), for $\ox$ to be a strong local minimizer of $\varphi$ with modulus $\sigma>0$ means the existence of $\epsilon>0$ such that 
$$
\varphi(x) \geq \varphi(\ox) + \frac{\sigma}{2}\|x-\ox\|^2 \quad \text{for all }\; x \in B_\epsilon(\ox).
$$
By construction \eqref{shift} of the $\sigma$-shift, the latter says that
$$
\psi(x) \geq \psi(\ox)  \quad \text{for all }\; x \in B_\epsilon(\ox),
$$
which means by Definition~\ref{minimizers}(i) that $\ox$ is a local minimizer of $\psi$.
\end{proof}

Now we are ready to establish the quantitative relationships between strong local minimizers for l.s.c.\ functions on Hilbert spaces and those for their Moreau envelopes. 

\begin{Theorem}[\bf strong local minimizers of Moreau envelopes]\label{localstrongMoreau}  Let $\varphi: X \to \overline{\R}$ be a proper, prox-bounded, and l.s.c. function defined on a Hilbert space $X$, and let $\ox\in \dom \varphi$. Consider the following  assertions:

{\bf(i)} $\ox$ is a strong local minimizer of $\varphi$ with modulus $\sigma>0$.

{\bf(ii)} $\ox$ is a strong local minimizer of $e_\lambda\varphi$ with modulus $\frac{\sigma}{1+\sigma\lambda}$ for all  $\lambda>0$ sufficiently small.\\
Then we have the implication {\rm(i)}$\Longrightarrow${\rm(ii)} when $\varphi$ is weakly sequentially l.s.c.\ The reverse implication {\rm(ii)}$\Longrightarrow${\rm(i)} holds if $P_\lambda \varphi(\ox)\neq \emptyset$ for all small $\lambda>0$; in particular, when $\varphi$ is weakly sequentially lower semicontinuous on $X$. 
\end{Theorem}
\begin{proof}
We start with verifying implication (i)$\Longrightarrow$(ii) under the weak sequential l.s.c.\ of $\varphi$. Considering the shifted function $\psi$ from \eqref{shift}, it follows from Proposition~\ref{localminishift} that $\ox$ is  a local minimizer of $\psi$. Theorem~\ref{localMoreau} tells us that  $\ox$ is a local minimizer of $e_\lambda\psi$ for all small $0<\lambda <\sigma$. Fix such a number $\lambda$ and define $\gamma:=\lambda/(1+\sigma\lambda)$. Since $\gamma<\lambda$, we find $\eta>0$ ensuring that 
\begin{equation}\label{localminofpsi}
e_\gamma \psi(x) \geq e_\gamma \psi(\ox) \quad \text{for all }\; x \in B_\eta(\ox). 
\end{equation}
Denoting $\epsilon:=(1+\sigma\lambda)\eta>0$ gives us the implication
\begin{equation}\label{shrinkBall}
x\in B_{\epsilon}(\ox )\Longrightarrow\frac{x+\lambda\sigma \ox}{1+\sigma\lambda}\in B_\eta(\ox ).
\end{equation}
By Lemma~\ref{aPhat}, we have the equality
\begin{equation}\label{envephiandpsi}
e_{\lambda}\varphi (x) = e_{\gamma}\psi \left( \dfrac{x+\sigma\lambda \ox }{1+\sigma\lambda}\right)+\dfrac{\sigma}{2(1+\sigma\lambda)}\|x-\ox\|^2,\quad x\in X.
\end{equation}
Combining \eqref{localminofpsi}, \eqref{shrinkBall}, and \eqref{envephiandpsi} brings us to
\begin{equation}\label{strongenvpsi}
e_\lambda \varphi(x) \geq e_\gamma \psi(\ox) + \dfrac{\sigma}{2(1+\sigma\lambda)}\|x-\ox\|^2 \quad \text{for all }\; x \in B_\epsilon(\ox). 
\end{equation}
Furthermore, by Lemma \ref{aPhat} we have the relationships
$$
e_\gamma \psi(\ox) = e_{\gamma/(1-\sigma\gamma)} \varphi \left(\frac{\ox-\sigma\gamma\ox}{1-\sigma\gamma} \right) =e_\lambda \varphi(\ox), 
$$
which being combined with \eqref{strongenvpsi} justify (ii). 
\vspace*{0.03in} 

Next we verify implication (ii)$\Longrightarrow$(i) under the additional assumption that $P_{\lambda}\varphi(\ox)\neq\emptyset$ for all small $\lambda>0$. Suppose that $\ox$ is a strong local minimizer of the Moreau envelope $e_\lambda \varphi$ with modulus $\mu(\lambda):=\sigma/(1+\sigma\lambda)$ for all $\lambda\in(0,\sigma^{-1})$ sufficiently small. Fixing such a number $\lambda$, for each $\gamma\in(0,\lambda/(1+\sigma\lambda))$ we have $0<\gamma/(1-\sigma\gamma)<\lambda$. This allows us to choose $\eta>0$ so that
\begin{align}\label{localphimin}
e_{\gamma/(1-\sigma\gamma)}\varphi (x) & \geq    e_{\gamma/(1-\sigma\gamma)}\varphi(\ox) + \frac{\mu(\gamma/(1-\sigma\gamma))}{2}\|x-\ox\|^2\nonumber \\
&=e_{\gamma/(1-\sigma\gamma)}\varphi(\ox) + \frac{\sigma (1-\sigma\gamma)}{2}\|x-\ox\|^2 \quad  \text{for all }\;x\in B_\eta  \left(\ox \right).
\end{align}
Letting $\epsilon:=(1-\sigma\gamma)\eta>0$ yields the implication
\begin{equation}\label{shrinkingball2}
x\in B_{\epsilon}(\ox )\Longrightarrow\frac{x-\sigma\gamma\ox}{1-\sigma\gamma}\in B_\eta \left(\ox \right).
\end{equation}
By Lemma~\ref{aPhat}, we have the representation
\begin{equation}\label{envephiandpsi2}
e_\gamma\psi(x)=e_{{\gamma}/(1-\sigma\gamma)}\varphi\left(\frac{x -\sigma\gamma \ox}{1-\sigma\gamma}\right)-\frac{\sigma}{2(1-\sigma\gamma)}\|x-\ox \|^2,\quad x\in X.
\end{equation}
Combining \eqref{localphimin}, \eqref{shrinkingball2}, and \eqref{envephiandpsi2} tells us that 
$$
e_\gamma\psi (x) \geq e_{\gamma/(1-\sigma\gamma)}\varphi(\ox) + \frac{\sigma (1-\sigma\gamma)}{2}\left\|\frac{x -\sigma\gamma \ox}{1-\sigma\gamma} -\ox\right\|^2-\frac{\sigma}{2(1-\sigma\gamma)}\|x-\ox \|^2
$$
for all $x\in B_\ve(\ox)$. In other words, we get
\begin{equation}\label{miniofpsienv}
e_\gamma\psi (x) \geq e_{\gamma/(1-\sigma\gamma)}\varphi(\ox)\;\text{ whenever }\; x \in B_\epsilon(\ox). 
\end{equation} 
Employing again Lemma~\ref{aPhat} gives us the equality
$$
e_{\gamma/(1-\sigma\gamma)}\varphi(\ox) = e_{\gamma} \psi \left(\frac{\ox+\sigma\gamma\ox}{1+\sigma\gamma} \right) =e_\gamma \psi(\ox), 
$$
which ensures together with \eqref{miniofpsienv} that $\ox$ is a local minimizer of $e_\gamma\psi$. This implies by Theorem~\ref{localMoreau}  that $\ox$ is also a local minimizer of $\psi$. Using finally Proposition~\ref{localminishift}, we deduce that $\ox$ is a strong local minimizer of $\varphi$ with modulus $\sigma$ and therefore complete the proof. 
\end{proof}
 
\section{Open Questions}\label{sec:open}

The paper establishes close relationships and equivalences between arbitrary local minimizers of general extended-real-valued functions and their Moreau envelopes in reflexive Banach spaces. Moreover, such relationships and equivalences are obtained for strong local minimizers of functions and their Moreau envelopes in Hilbert spaces. This line of research should be continued, since many important unsolved questions remain. We now list just few of them.\\

$\bullet$  Is it true that Moreau envelopes preserve the other types of local minimizers such as   weak-sharp minimizers, tilt-stable minimizers in infinite-dimensions, etc.? \\

$\bullet$ How to apply the obtained results to construct second-order necessary and sufficient conditions for nonsmooth optimization problems? \\

$\bullet$ How to apply the obtained results to conduct local convergence analysis of first-order and second-order numerical optimization algorithms in finite- and infinite-dimensional spaces?\\[2ex]
\noindent
We plan to investigate these and related topics in our future research.

\end{document}